\documentclass[english, a4paper, 12pt]{article}
\usepackage{multicol,amsmath,t1enc,a4,babel}
\usepackage{amsthm}
\usepackage{amssymb}
\usepackage[dvips]{epsfig}
\usepackage{times}
\usepackage{amsfonts}
\usepackage{wasysym}
\usepackage{mathtools}
\usepackage[colorinlistoftodos]{todonotes}
\usepackage{mathrsfs}
\usepackage{relsize}
\usepackage{bbm}
\usepackage{graphicx}
\usepackage{subcaption}

\newcommand{\Z}{\mathbb{Z}}
\newcommand{\N}{\mathbb{N}}

\newtheorem{defi}{Definition}[section]
\newtheorem{lemma}[defi]{Lemma}
\newtheorem{prop}[defi]{Proposition}
\newtheorem{theorem}[defi]{Theorem}

\newtheorem{rem}[defi]{Remark}

\begin{document}

\title{Note on AR(1)-characterisation of stationary processes and model fitting}

\renewcommand{\thefootnote}{\fnsymbol{footnote}}

\author{Marko Voutilainen\footnotemark[1]\, and \, Lauri Viitasaari\footnotemark[2]\, and \, Pauliina Ilmonen\footnotemark[1]}

\footnotetext[1]{Department of Mathematics and Systems Analysis, Aalto University School of Science, Finland}

\footnotetext[2]{Department of Mathematics and Statistics, University of Helsinki, 
Finland}

\maketitle

\begin{abstract}
\noindent
It was recently proved that any strictly stationary stochastic process can be viewed as an autoregressive process of order one with coloured noise. Furthermore, it was proved that, using this characterisation, one can define closed form estimators for the model parameter based on autocovariance estimators for several different lags. However, this estimation procedure may fail in some special cases. In this article we provide a detailed analysis of these special cases. In particular, we prove that these cases correspond to degenerate processes.
\end{abstract}

{\small
\medskip


\medskip

}

%%%%%%%%%%%%%%%%%%%%%%%%%%%%%%%%%%%%%%%%%%%%%%%%%%%%%%%%%%%%%%%%%%%%%%%%%%%%%%%

%\tableofcontents

%%%%%%%%%%%%%%%%%%%%%%%%%%%%%%%%%%%%%%%%%%%%%%%%%%%%%%%%%%%%%%%%%%%%%%%%%%%%%%%

\section{Introduction}
Stationary processes are important tool in many practical applications of time series analysis, and the topic is extensively studied in the literature. Traditionally, stationary processes are modelled by using autoregressive moving average processes or linear processes (see monographs \cite{brockwell,hamilton1994time} for details). 

One of the most simple example of an autoregressive moving average process is an autoregressive process of order one. That is, a process $(X_t)_{t\in\Z}$ defined by  
\begin{equation}
\label{eq:ar1-intro-standard}
X_t = \phi X_{t-1} + \varepsilon_t,\quad t\in\Z,
\end{equation}
where $\phi\in(-1,1)$ and $(\varepsilon_t)_{t\in\Z}$ is a sequence of independent and identically distributed square integrable random variables. The continuous time analogue of \eqref{eq:ar1-intro-standard} is called the Ornstein-Uhlenbeck process, which can be defined as the stationary solution of the Langevin-type stochastic differential equation 
\begin{equation}
\label{eq:langevin-intro}
d U_t = - \phi U_t dt + dW_t,
\end{equation}
where $\phi>0$ and $(W_t)_{t\in \mathbb{R}}$ is a two-sided Brownian motion. Such equations have also applications in mathematical physics.

Statistical inference for AR(1)-process or Ornstein-Uhlenbeck process is well-established in the literature.  Furthermore, recently a generalised continuous time Langevin equation, where the Brownian motion $W$ in \eqref{eq:langevin-intro} is replaced with a more general driving force $G$, have been a subject of active study. Especially, the so-called fractional Ornstein-Uhlenbeck processes introduced by \cite{Cheridito-Kawaguchi-Maejima-2003} have been studied extensively. For parameter estimation in such models, we mention a recent monograph \cite{julia} dedicated to the subject, and the references there in.

When the model becomes more complicated, the number of parameters increases and the estimation may become a challenging task. For example, it may happen that standard maximum likelihood estimators cannot be expressed in closed form \cite{brockwell}. Even worse, it may happen that classical estimators such as maximum likelihood or least squares estimators are biased and not consistent (cf. \cite{sole-et-al} for discussions on the generalised ARCH-model with fractional Brownian motion driven liquidity). One way to tackle such problems is to consider one parameter model, and to replace white noise in \eqref{eq:ar1-intro} with some other stationary noise. It was proved in \cite{vouti} that each discrete time strictly stationary process 
can be characterised by
\begin{equation}
\label{eq:ar1-intro}
X_t = \phi X_{t-1} + Z_t,
\end{equation}
where $\phi\in(0,1)$. This representation can be viewed as a discrete time analogue of the fact that Langevin-type equation characterises strictly stationary processes in continuous time \cite{viitasaari-ss}. 

The authors in \cite{vouti} applied characterisation \eqref{eq:ar1-intro} to model fitting and parameter estimation. The presented estimation procedure is straightforward to apply with the exception of certain special cases.
 The purpose of this paper is to provide a comprehensive analysis of these special cases. In particular, we show that such cases do not provide very useful models. This highlights the wide applicability of characterization \eqref{eq:ar1-intro} and the corresponding estimation procedure. 
 
The rest of the paper is organised as follows. In Section \ref{sec:results} we briefly discuss the motivating estimation procedure of \cite{vouti}. We also present and discuss our main results together with some illustrative figures. All the proofs and technical lemmas are postponed to Section \ref{sec:proofs}.

\section{Motivation and formulation of the main results}
\label{sec:results}
Let $X = (X_t)_{t\in\Z}$ be a stationary process. It was shown in \cite{vouti} that equation
\begin{equation}
\label{eq:ar1}
X_t = \phi X_{t-1} + Z_t,
\end{equation}
where $\phi\in(0,1)$ and $Z_t$ is another stationary process, characterises all discrete time (strictly) stationary processes. Throughout this paper we suppose that $X$ and $Z$ are square integrable processes with autocovariance functions $\gamma(\cdot)$ and $r(\cdot)$, respectively. Using Equation \eqref{eq:ar1}, one can derive Yule-Walker type equations for the parameter $\phi$, which can be solved in an explicit form. Namely, for any $m\in \Z$ such that $\gamma(m)\neq 0$ we have
\begin{equation}
\label{eq:phi}
\phi = \frac{\gamma(m+1)+\gamma(m-1)\pm \sqrt{(\gamma(m+1)+\gamma(m-1))^2-4\gamma(m)(\gamma(m)-r(m))}}{2\gamma(m)}.
\end{equation}
The estimation of the parameter $\phi$ is obvious from \eqref{eq:phi} provided that 
one can determine which sign, plus or minus, one should choose. In practice, this can be done by choosing different lags $m$ for which to estimate the covariance function $\gamma(m)$. Then one can determine the correct value $\phi$ by comparing different signs in \eqref{eq:phi} for different lags $m$ (We refer to \cite[p. 387]{vouti} for detailed discussion). However, this approach fails, i.e. one cannot find suitably chosen lags leading to the correct choice of the sign and only one value $\phi$, if, for $m\in\Z$ such that $\gamma(m)=0$ we also have $r(m)=0$, and for any $m\in \Z$ such that $\gamma(m)\neq 0$, the ratio
\begin{equation}
\label{eq:r-gamma-ratio}
a_m = \frac{r(m)}{\gamma(m)} = a
\end{equation}
for some constant $a\in(0,1)$. The latter is equivalent \cite[p. 387]{vouti} to the fact that 
\begin{equation}
\label{eq:gamma-ratio}
\frac{\gamma(m+1)+\gamma(m-1)}{\gamma(m)} = b
\end{equation}
for some constant $b$ with $\phi<b<\phi+\phi^{-1}$. This leads to
\begin{equation}
\label{recursion}
\gamma(m+1) = b\gamma(m) - \gamma(m-1).
\end{equation}
Moreover, if $\gamma(m)=r(m)=0$ for some $m$, it is straightforward to verify that \eqref{recursion} holds in this case as well. Thus \eqref{recursion} holds for all $m\in\Z$. Since covariance functions are necessarily symmetric, we obtain an ''initial'' condition $\gamma(1)=\frac{b}{2}\gamma(0)$. Thus \eqref{recursion} admits a unique symmetric solution. 

From $\gamma(1)=\frac{b}{2}\gamma(0)$ it is clear that \eqref{recursion} does not define covariance function for $b>2$. Furthermore, since $\phi>0$, it suffices to study the regime $b\in[0,2]$ (we include the trivial case $b=0$). 
For $b=2$ this corresponds to the case $X_t = X_0$ for all $t\in \Z$ which is hardly interesting. Similarly, the case $b=0$ leads to a process $(\ldots,X_0,X_1,-X_0,-X_1,X_0,X_1,\ldots)$ which again does not provide a practical model. On the other hand, it is not clear whether for some other values $b\in(0,2)$ Equation \eqref{recursion} can lead to some non-trivial model in which estimation procedure explained above cannot be applied. It turns out that, for any $b\in[0,2]$, Equation \eqref{recursion} defines a covariance function. On the other hand, the resulting covariance function, denoted by $\gamma_b$,  leads to a model that is either not very interesting. 
\begin{theorem}
\label{thm:main}
Let $b\in(0,2)$ and $\gamma_b$ be the (unique) symmetric function satisfying \eqref{recursion}. Then
\begin{enumerate}
\item Let $b  = 2\sin{\left(\frac{k}{l}\frac{\pi}{2}\right)}$,
where $k$ and $l$ are strictly positive integers such 
that $\frac{k}{l}\in(0,1)$. Then $\gamma_b(m)$ is periodic.
\item Let $b  = 2\sin{\left(r\frac{\pi}{2}\right)}$, where $r\in(0,1)\setminus \mathbb{Q}$. Then for any $M\geq 0$, the set $\{\gamma_b(M+m):m\geq 0\}$ is dense in $[-\gamma(0),\gamma(0)]$.
\item For any $b\in[0,2]$, $\gamma_b(\cdot)$ is a covariance function.
\end{enumerate}
\end{theorem}
In many applications of stationary processes, it is assumed that the covariance function $\gamma(\cdot)$ vanishes at infinity, or that $\gamma(\cdot)$ is periodic. Note that the latter case corresponds simply to the analysis of finite-dimensional random vectors with identically distributed components. Indeed, $\gamma(m)=\gamma(0)$ implies $X_n = X_0$ almost surely, so periodicity of $\gamma(\cdot)$ with period $N$ implies that there exists at most $N$ random variables as the source of randomness. By items (2) and (3) of Theorem \ref{thm:main}, we observe that, for suitable values of $b$, \eqref{recursion} can be used to construct covariance functions that are neither periodic nor vanishing at infinity. On the other hand, in this case there are arbitrary large lags $m$ such that $\gamma_b(m)$ is arbitrary close to $\gamma_b(0)$. Consequently, it is expected that different estimation procedures fail. Indeed, even the standard covariance estimators are not consistent. A consequence of Theorem \ref{thm:main} is that only a little structure in the noise $Z$ is needed in order to apply the estimation procedure of the parameter $\phi$ introduced in \cite{vouti}, provided that one has consistent estimators for the covariances of $X$. The following is a precise mathematical formulation of this observation. 
\begin{theorem}
\label{thm:estimation}
Let $X$ be given by \eqref{eq:ar1} for some $\phi\in(0,1)$ and noise $Z$. Assume that there exists $\epsilon>0$ and $M\in\N$ such that $r(m) \leq r(0)(1-\epsilon)$ or $r(m)\geq -r(0)(1-\epsilon)$ for all $m\geq M$. Then the covariance function $\gamma$ of $X$ does not satisfy \eqref{recursion} for any $b\in[0,2]$. 
\end{theorem}
We end this section by visual illustrations of the covariance functions defined by \eqref{recursion}. In these examples we have set $\gamma_b(0) = 1$. In Figures \ref{fig:kandl} and \ref{fig:kandl2} we have illustrated the case of item (1) of Theorem \ref{thm:main}. Note that in Figure \ref{fig:1per3} we have $b=2\sin \left(\frac{1}{3}\frac{\pi}{2}\right) = 1$. Figure \ref{fig:kandl2} demonstrates how $k$ can affect the shape of the covariance function. Finally, Figure \ref{fig:1.7}  illustrates the case of item (2) of Theorem \ref{thm:main}.

\begin{figure}[h!]
\centering
  \begin{subfigure}[t]{0.49\textwidth}
   \includegraphics[width=\textwidth]{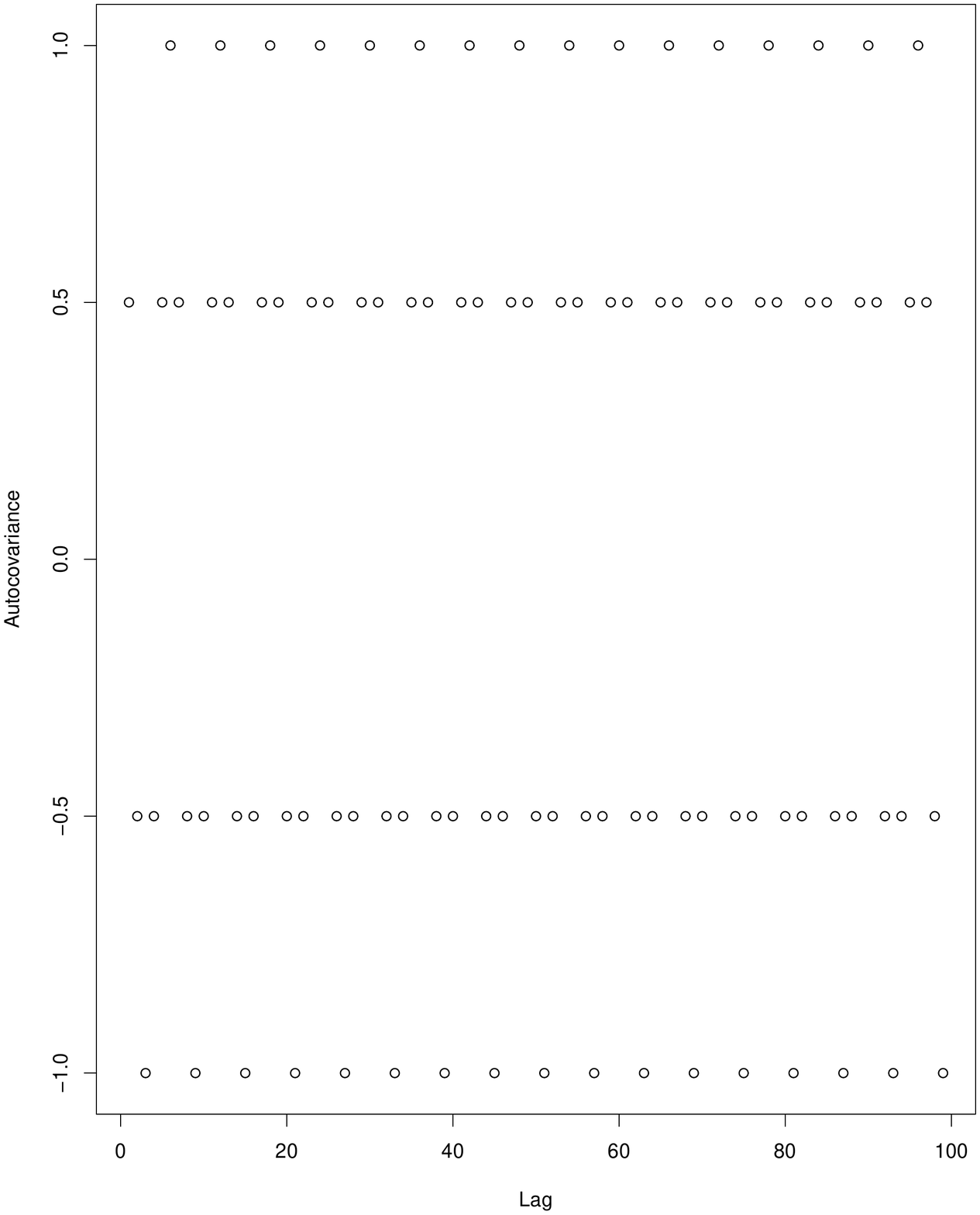}
    \caption{$k=1$ and $l=3$.}
    \label{fig:1per3}
  \end{subfigure}
 \begin{subfigure}[t]{0.49\textwidth}
    \includegraphics[width=\textwidth]{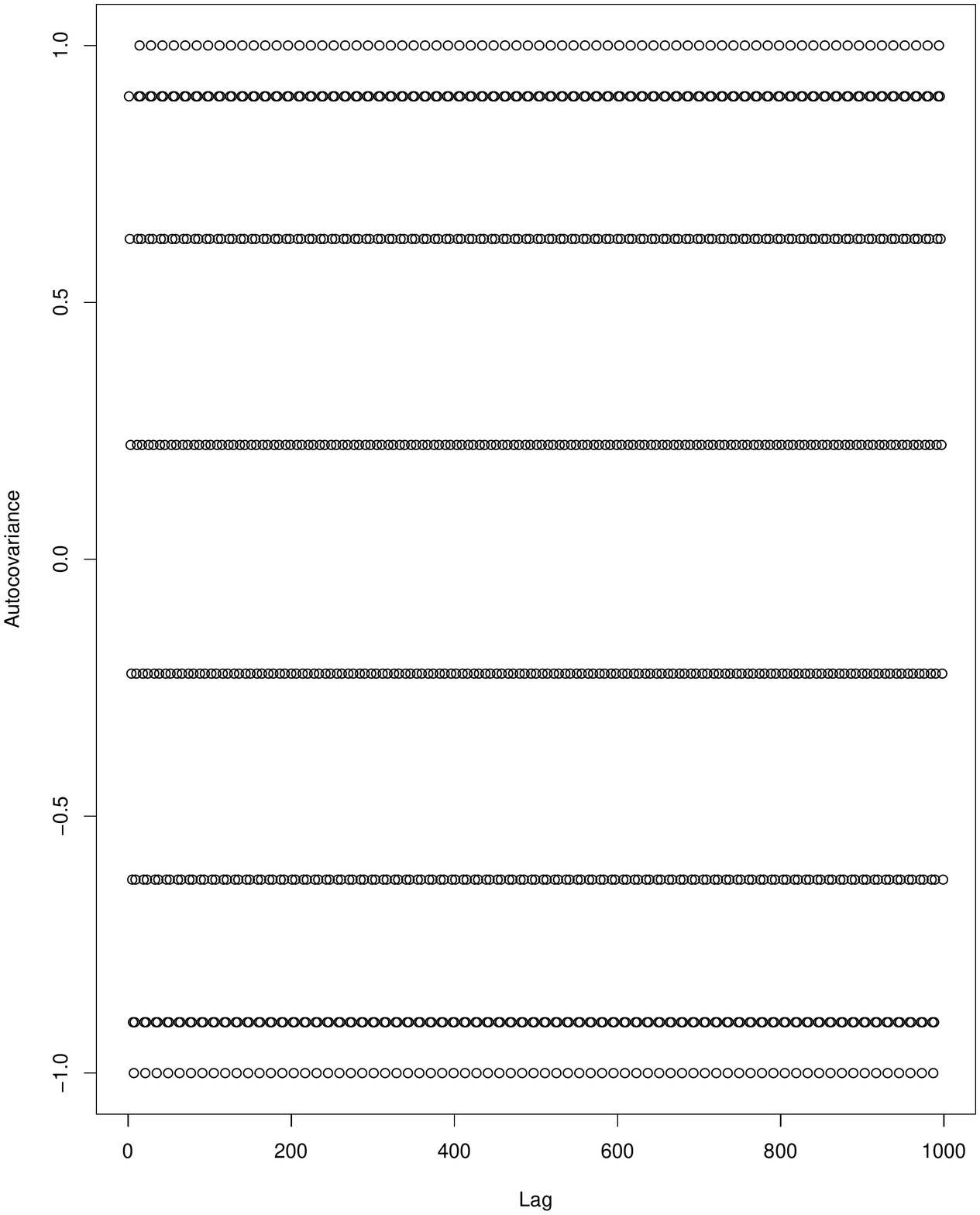}
    \caption{$k=5$ and $l=7$.}
   \label{fig:5per7}
 \end{subfigure}
\caption{Examples of covariance functions corresponding to $b = 2\sin \left(\frac{k}{l}\frac{\pi}{2}\right)$.}
\label{fig:kandl}
\end{figure}

\begin{figure}[h!]
\centering
  \begin{subfigure}[t]{0.49\textwidth}
   \includegraphics[width=\textwidth]{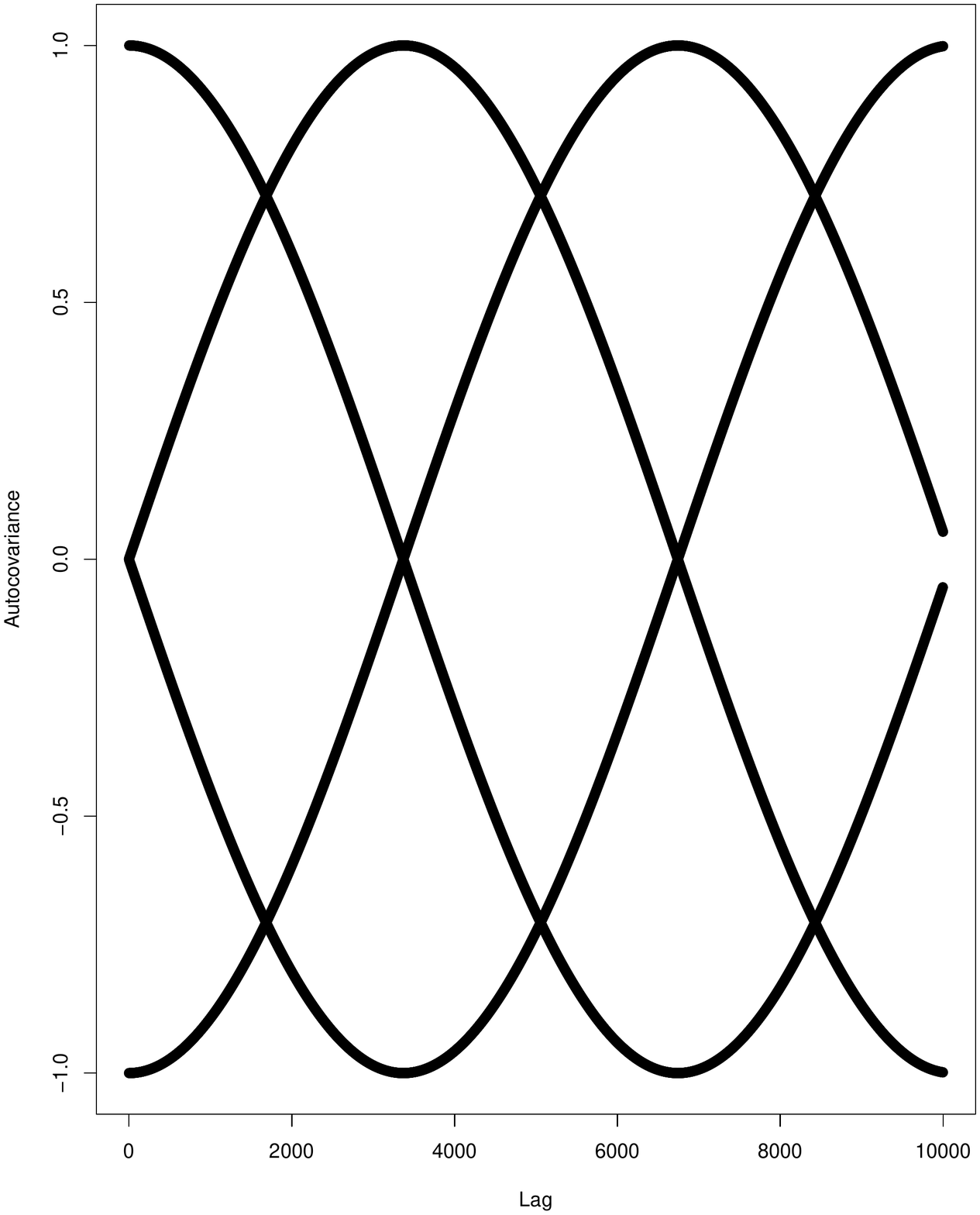}
    \caption{$k=1$ and $l=3371$.}
    \label{fig:1per3371}
  \end{subfigure}
 \begin{subfigure}[t]{0.49\textwidth}
    \includegraphics[width=\textwidth]{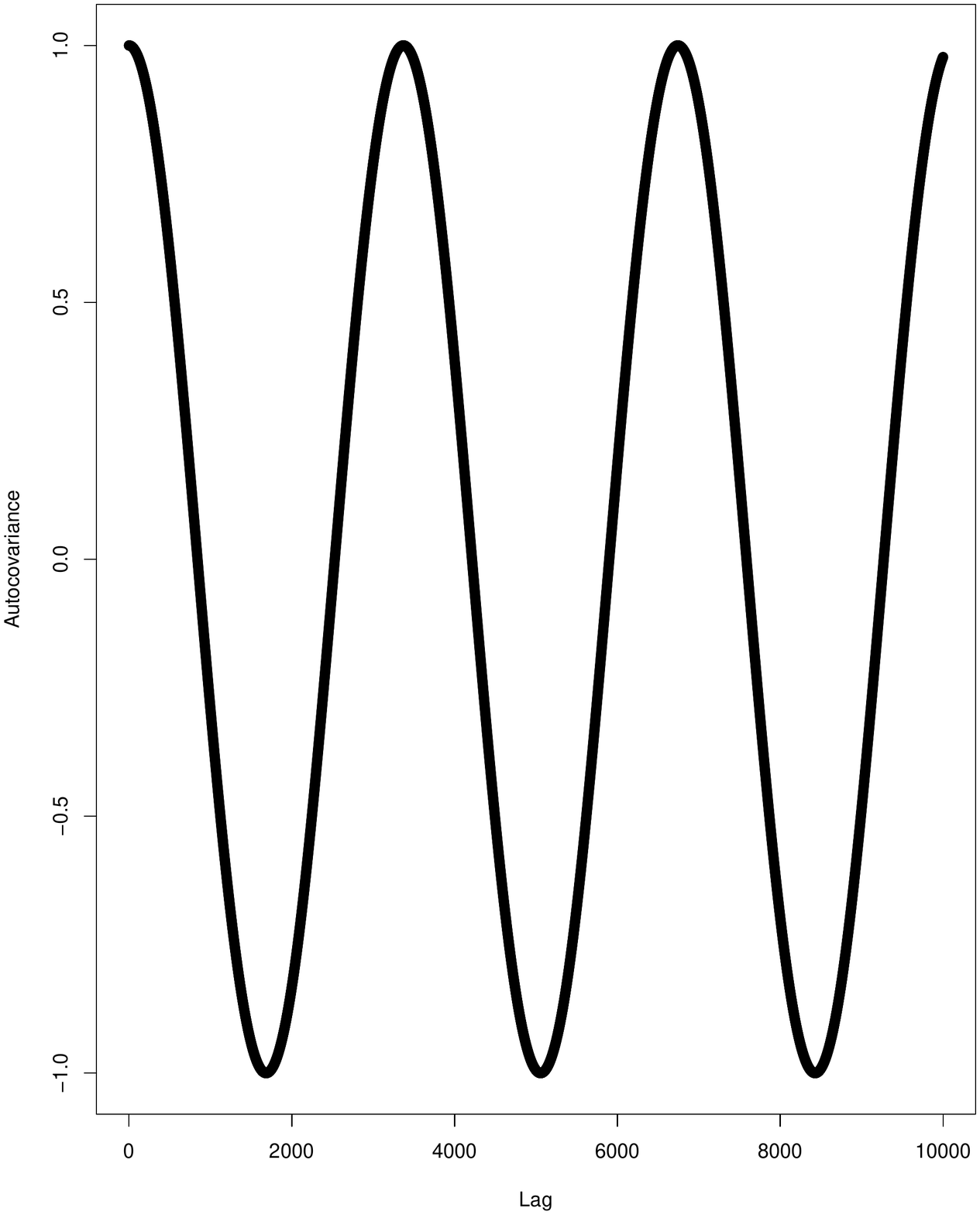}
    \caption{$k=3367$ and $l=3371$.}
   \label{fig:3367per3371}
 \end{subfigure}
\caption{Examples of covariance functions corresponding to $b = 2\sin \left(\frac{k}{l}\frac{\pi}{2}\right)$.}
\label{fig:kandl2}
\end{figure}

\begin{figure}[h!]
\centering
  \begin{subfigure}[t]{0.49\textwidth}
   \includegraphics[width=\textwidth]{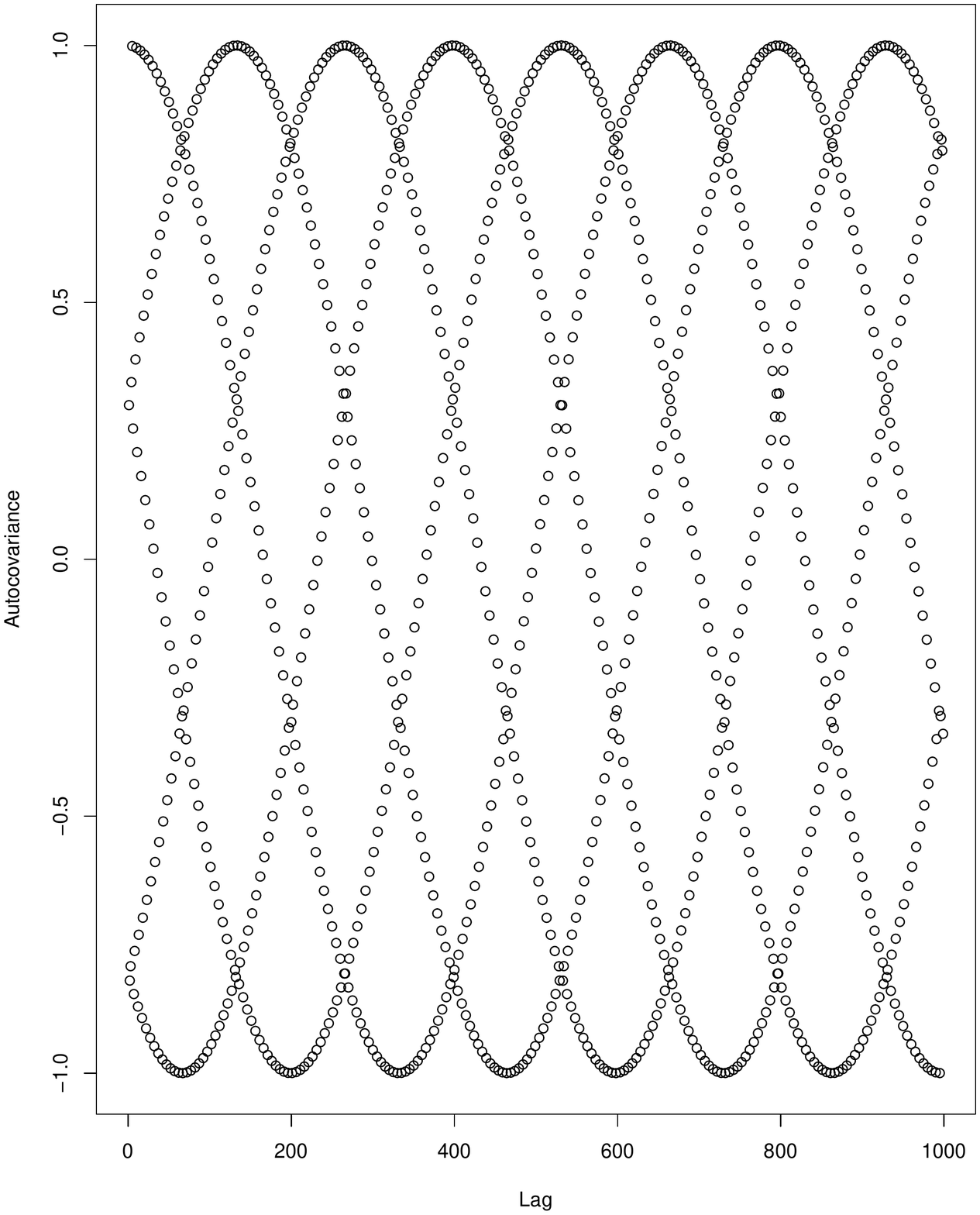}
    \caption{$b=0.6$.}
    \label{fig:0.6}
  \end{subfigure}
 \begin{subfigure}[t]{0.49\textwidth}
    \includegraphics[width=\textwidth]{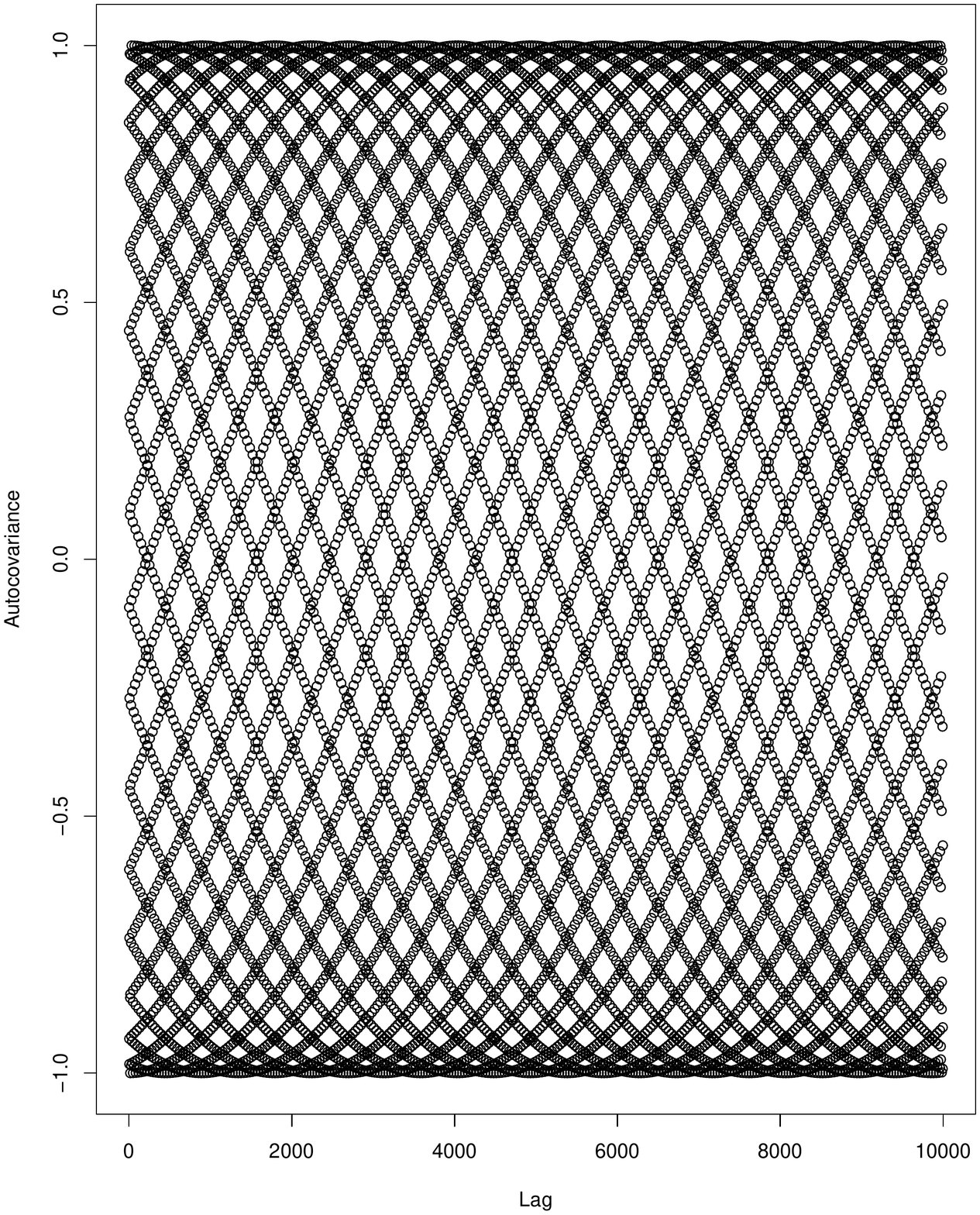}
    \caption{$b=1.7$.}
   \label{fig:1.7}
 \end{subfigure}
\caption{Examples of covariance functions corresponding to $b = 2\sin \left(r\frac{\pi}{2}\right)$.}
\label{fig:r}
\end{figure}

\section{Proofs}
\label{sec:proofs}
Throughout this section, without loss of generality, we assume $\gamma_b(0)=1$. We also drop the sub-index and simply denote $\gamma$. The following first result gives explicit formula for the solution to \eqref{recursion}.
\begin{prop}
\label{prop:solution}
The unique symmetric solution to \eqref{recursion} is given by
\begin{equation}
\label{solution}
\gamma(m) = \left\{
\begin{array}{ll}
(-1)^{\frac{m}{2}} \cos{\left(m\arcsin{\left(\frac{b}{2}\right)}\right)},\quad&\text{for $m$ is even}\\
(-1)^{\frac{(m-1)}{2}}  \sin{\left(m\arcsin{\left(\frac{b}{2}\right)}\right)},\quad&\text{for $m$ is odd}.
\end{array}
\right.
\end{equation}
\end{prop}
\begin{proof}
Clearly, $\gamma(m)$ given by \eqref{solution} is symmetric, and thus it suffices to consider $m\geq 0$. Moreover $\gamma(0) =1$ and $\gamma(1) = \frac{b}{2}$. We use the short notation $A = \arcsin{\left(\frac{b}{2}\right)}$ so that $\sin A = \frac{b}{2}$. Assume first $m+2 \equiv 2 \pmod{4}$. Then

\begin{equation*}
\begin{split}
\gamma(m+2) &= -\cos{\left((m+2)A\right)}\\ 
&= -\cos{(mA)}\cos{(2A)} + \sin{(mA})\sin{(2A)}\\
&= -\cos{(mA)}\left(1-2\sin^2A\right) + 2\sin{(mA})\sin{A}\cos{A}\\
&= -\cos{(mA)}\left(1-b\sin{A}\right)+ b\sin{(mA)}\cos{A}\\
&= b\left(\cos{(mA)}\sin{A} + \sin{(mA)}\cos{A}\right) - \cos{(mA)}\\
&= b\sin{\left((m+1)A\right)} - \cos{(mA)} \\
&=b\gamma(m+1) - \gamma(m).
\end{split}
\end{equation*}
Similarly, for $m+2 \equiv 3 \pmod{4}$ we observe

\begin{equation*}
\begin{split}
\gamma(m+2) &= -\sin{((m+2)A)}\\
 &= -\sin{(mA)}\cos{(2A)} - \sin{(2A)}\cos{(mA)}\\
&= -\sin{(mA)}(1- 2\sin^2A)- 2\sin{A}\cos{A}\cos{(mA)}\\
&= -b(\cos{A}\cos{(mA)} - \sin{(mA)}\sin{A}) - \sin{(mA)}\\
&= -b\cos{((m+1)A)} - \sin{(mA)}\\
&= b\gamma(m+1) - \gamma(m).
\end{split}
\end{equation*}
Treating cases $m+2 \equiv 0 \pmod{4}$ and $m+2 \equiv 1 \pmod{4}$ similarly, we deduce that \eqref{solution} satisfies \eqref{recursion}.
\end{proof}

\begin{rem}

Using \eqref{recursion} directly, we observe, for even $m\geq 1$, that 
\begin{footnotesize}
\begin{equation}
\label{even}
\gamma(m) = b^ m + \sum_{n=\frac{m}{2}}^{m-1} (-1)^ {m-n} \left(\binom{n}{m-n} b^{2n-m} + \binom{n}{m-n-1}\frac{b^{2n-m+2}}{2}\right).
\end{equation}
\end{footnotesize}
Similarly, for odd $m\geq 1$, we obtain
\begin{footnotesize}
\begin{equation}
\label{odd}
\gamma(m) = \sum_{n=\frac{m+1}{2}}^{m} (-1)^ {m-n} \binom{n}{m-n} b^{2n-m} + \sum_{n=\frac{m-1}{2}}^{m-1} (-1)^ {m-n} \binom{n}{m-n-1} \frac{b^{2n-m+2}  }{2}.
\end{equation}
\end{footnotesize}
These formulas are finite polynomial expansions, in variable $b$, of the functions presented in \eqref{solution} which could have been deduced also by using some well-known trigonometric identities.
\end{rem}
Before proving our main theorems we need several technical lemmas.
\begin{defi}
We denote with $Q$ a subset of rationals defined by 
\begin{equation*}
Q \coloneqq \left\{\frac{k}{l} : k,l \in \mathbb{N}, \frac{k}{l}\in (0,1), k - l\equiv 1 \pmod 2  \right\}
\end{equation*}
\end{defi}
\begin{rem}
The modulo condition above means only that either $k$ is even and $l$ is odd, or vice versa.
\end{rem}

\begin{lemma}
\label{lemma:sumcos}
Let $x=\frac{k}{l}\frac{\pi}{2}$, where $\frac{k}{l}\in Q$. Then

\begin{equation*}
\sum_{j=1}^ {2l-1} \cos^ 2{(jx)}(-1)^ j = -1.
\end{equation*}
\end{lemma}
\begin{proof}
We write

\begin{equation*}
\sum_{j=1}^ {2l-1} \cos^ 2{(jx)}(-1)^ j  = \cos^2{(lx)}(-1)^ l + \sum_{j=1}^ {l-1} \cos^ 2{(jx)}(-1)^j + \sum_{j=l+1}^ {2l-1} \cos^ 2{(jx)}(-1)^ j.
\end{equation*}
Change of variable $t=j-l$ gives

\begin{equation*}
\begin{split}
\sum_{j=l+1}^ {2l-1} \cos^ 2{(jx)}(-1)^ j &= \sum_{t=1}^{l-1} \cos^ 2{((t+l)x)}(-1)^ {t+l}\\
=  \sum_{t=1}^{l-1} \cos^ 2{(tx+k\frac{\pi}{2})}(-1)^ {t+l}
&= \left\{
\begin{array}{ll}
\sum_{t=1}^ {l-1} \cos^ 2{(tx)}(-1)^ {t+l}, \quad k\text{ even} \\ [0.1cm]
\sum_{t=1}^ {l-1} \sin^ 2{(tx)}(-1)^ {t+l}, \quad k\text{ odd}.
\end{array}
\right.
\end{split}
\end{equation*}
Consequently, for even $k$ and odd $l$ we have

\begin{equation*}
\sum_{j=1}^ {2l-1} \cos^ 2{(jx)}(-1)^ j = -\cos^2{\left(k\frac{\pi}{2}\right)} = -1.
\end{equation*}
Similarly, for odd $k$ and even $l$,

\begin{equation*}
\sum_{j=1}^ {2l-1} \cos^ 2{(jx)}(-1)^ j = \cos^2{\left(k\frac{\pi}{2}\right)} + \sum_{j=1}^{l-1}(-1)^ j = -1.
\end{equation*}
\end{proof}

\begin{lemma}
\label{lemma:eigenvalues}
Let $\gamma(\cdot)$ be given by \eqref{solution} with $
b  = 2\sin{\left(\frac{k}{l}\frac{\pi}{2}\right)}
$
for some $\frac{k}{l}\in Q$. Then the non-zero eigenvalues of the matrix

\begin{equation}
\label{eq:covariance-matrix}
C \coloneqq
\begin{bmatrix}
\gamma(0) & \gamma(1) & \gamma(2)&\cdots& \gamma(4l-1)\\
\gamma(1) & \gamma(0) &\gamma(1)&\cdots & \gamma(4l-2)\\
\gamma(2) & \gamma(1) & \gamma(0)& \cdots & \gamma(4l-3)\\
\vdots& \vdots&\vdots&\ddots&\vdots\\
\gamma(4l-2)&\gamma(4l-3)&\gamma(4l-4)&\cdots&\gamma(1)\\
\gamma(4l-1)& \gamma(4l-2)&\gamma(4l-3)&\cdots&\gamma(0)
\end{bmatrix}
\end{equation}
are either $2l$ with multiplicity of two or $4l$ with multiplicity of one. 
\end{lemma}
\begin{proof}
Let $c_i$ denote the $i$th column of $C$. Then, by the defining equation \eqref{recursion}, $c_i = bc_{i-1}- c_{i-2}$ for any $i\geq 3$. Consequently, there exists at most two linearly independent columns. Thus $rank(C)\leq 2$, which in turn implies that there exists at most two non-zero eigenvalues $\lambda_1$ and $\lambda_2$. In order to compute $\lambda_1$ and $\lambda_2$, we recall the following identities:

\begin{align}
 tr(C) &= \lambda_1 + \lambda_2 = 4l \label{trace1}\\
 tr(C^2) &= \lambda_1^2 + \lambda_2^2 = ||C||_F^2 \label{trace2},
\end{align}
where $||\cdot||_F$ is the Frobenius norm. If $rank(C)=1$, then $\lambda_2=0$ implying the second part of the claim. Suppose then $rank(C)=2$. Observing that the squared sum of the diagonals is $4l$ and, for $j=1,2,\ldots,4l-1$, a term $\gamma(j)$ appears in $C$ exactly $2(4l-j)$ times, we obtain
\begin{equation*}
\begin{split}
||C||_F^ 2 &= 4l + 2\sum_{j=1}^{4l-1} (4l-j)\gamma(j)^ 2.
\end{split}
\end{equation*}
Dividing the sum into two parts and using $\sin^ 2(x) = 1- \cos^2(x)$ we have
\begin{equation*}
\begin{split}
||C||_F^ 2 &= 4l+ 2\sum_{j=0}^{2l-1} (4l-(2j+1)) \gamma(2j+1)^ 2 + 2\sum_{j=1}^ {2l-1}(4l-2j)\gamma(2j)^ 2\\
&=4l + 2\sum_{j=0}^{2l-1} (4l - (2j+1))\sin^2{((2j+1)x)} + 2\sum_{j=1}^ {2l-1} (4l-2j)\cos^2{(2jx)}\\
&=4l + 2\sum_{j=0}^{2l-1} (4l - (2j+1)) + 2\sum_{j=1}^ {4l-1} (4l-j)\cos^2{(jx)}(-1)^j\\
&=8l^2 + 4l + 2\sum_{j=1}^ {4l-1} (4l-j)\cos^2{(jx)}(-1)^j,
\end{split}
\end{equation*}
where in the last equality we have used
\begin{equation*}
\sum_{j=0}^{2l-1} (4l - (2j+1)) = \sum_{j=0}^ {2l-1} (4l-1) - 2\sum_{j=0}^{2l-1} j = 2l(4l-1) + 2l(2l-1) = 4l^ 2.
\end{equation*}
Now 
\begin{equation}
\label{frobenius2}
\begin{split}
\sum_{j=1}^ {4l-1} (4l-j)\cos^2{(jx)}(-1)^j &= 2l+ \sum_{j=1}^{2l-1} (4l-j)\cos^2(jx)(-1)^ j\\
&\ + \sum_{j=2l+1}^ {4l-1} (4l-j)\cos^2(jx)(-1)^ j,
\end{split}
\end{equation}
where substitution $j = 4l-t$ yields

\begin{equation}
\label{frobenius3}
\begin{split}
\sum_{j=2l+1}^ {4l-1} (4l-j)\cos^2(jx)(-1)^ j &= \sum_{t=1}^ {2l-1} t \cos^ 2((4l-t)x)(-1)^{4l-t}\\
= \sum_{t=1}^ {2l-1} t\cos^2(2k\pi-tx)(-1)^t
&= \sum_{t=1}^ {2l-1} t\cos^2(tx)(-1)^t.
\end{split}
\end{equation}
Now \eqref{frobenius2}, \eqref{frobenius3}, and Lemma \ref{lemma:sumcos} imply
\begin{equation*}
||C||_F^ 2 = 8l^2+4l+2\left(2l+ 4l\sum_{j=1}^{2l-1} \cos^2(jx)(-1)^j\right) = 8l^ 2.
\end{equation*}
Finally, using \eqref{trace1} and \eqref{trace2} together with $||C||_F^2 = 8l^2$, we obtain
\begin{equation*}
\begin{split}
 \lambda_1^ 2 + (4l-\lambda_1)^ 2  - 8l^2= 2\lambda_1^ 2-8l\lambda_1+ 8l^2 = (\sqrt{2}\lambda_1 - \sqrt{8}l)^2 =0.
\end{split}
\end{equation*}
Hence $\lambda_1=\lambda_2=2l$. 
\end{proof}

We are now ready to prove Theorem \ref{thm:main} and Theorem \ref{thm:estimation}.

\begin{proof}[Proof the Theorem \ref{thm:main}]
Throughout the proof we denote $a_2 \equiv a_1 \pmod{2\pi}$ if $a_2 = a_1 + 2k\pi$ for some $k\in\mathbb{Z}$. That is, $a_1$ and $a_2$ are identifiable when regarding them as points on the unit circle. By $a_3 \in (a_1, a_2) \pmod{2\pi}$ we mean that $a_3 \equiv a \pmod{2\pi}$ for some $a\in (a_1,a_2)$.
\begin{enumerate}
\item 
Since
$
\arcsin{\left(\frac{b}{2}\right)} = \frac{k}{l}\frac{\pi}{2},
$
the first claim follows from Proposition \ref{prop:solution} together with the fact that functions $\sin(\cdot)$ and $\cos(\cdot)$ are periodic. In particular, we have $\gamma(4l+m) = \gamma(m)$ for every $m\in\Z$. 
\item 
Denote $A = \arcsin\left(\frac{b}{2}\right) = r\frac{\pi}{2}$. By Proposition \ref{prop:solution}, $mA$ is the corresponding angle for $\gamma(m)$ on the unit circle. Note first that, due the periodic nature of $\cos$ and $\sin$ functions, it suffices to prove the claim only in the case $M=0$. In what follows, we assume that $m\geq 0$. We show that the function $\gamma(m),m\equiv 0\pmod{4}$ is dense in $[-1,1]$, while a similar argument could be used for other equivalence classes as well. That is, we show that the function $cos(mA),m\equiv 0\pmod{4} $ is dense in $[-1,1]$. Essentially this follows from the observation that, as $r\notin\mathbb{Q}$, the function $m \mapsto cos(mA)$ is injective. Indeed, if $\cos(\tilde{m}A) = cos(mA)$ for some $\tilde{m}, m \geq 0$, $\tilde{m} \neq m$, it follows that 
\begin{equation*}
\tilde{m}A = \tilde{m}r\frac{\pi}{2} =\pm  mr\frac{\pi}{2} + k2\pi =\pm mA + k2\pi\qquad\text{for some }k\in\mathbb{Z}.
\end{equation*}
This implies 
\begin{equation*}
r = \frac{4k}{\tilde{m}\pm m},
\end{equation*}
which contradicts $r\notin\mathbb{Q}$. Since $\cos(mA)$ is injective, it is intuitively clear that $\cos(mA),m\equiv 0\pmod{4}$ is dense in $[-1,1]$. For a precise argument, we argue by contradiction and assume there exists an interval $(c_1, d_1) \subset [-1,1]$ such that $\cos(mA) \notin (c_1, d_1)$ for any $m\equiv 0\pmod{4}$. This implies that there exists an interval $(c_2, d_2) \subset [0, 2\pi]$ such that for every $m\equiv 0\pmod{4}$ it holds that $mA \notin (c_2,d_2)\pmod{2\pi}$. Without loss of generality, we can assume $c_2=0$ and that for some $m_0 \equiv 0\pmod{4}$ we have $m_0A \equiv 0 \pmod{2\pi}$. Let $m_n = m_0 + 4n$ with $n\in \N$ and denote by $\lfloor\cdot\rfloor$ the standard floor function. Suppose that for some $n\in\N$ and $p_n\in (-d_2,0)$ we have $m_nA \equiv p_n \pmod{2\pi}$. Since by injectivity $\frac{2\pi}{|p_n|}\notin\N$, we get $m_{n\lfloor\frac{2\pi}{|p_n|}\rfloor}A \in (0,d_2) \pmod{2\pi}$ leading to a contradiction. This implies that for every $n\in\N$ we have $m_nA\notin(-d_2,d_2)\pmod{2\pi}$ (for a visual illustration, see Figure \ref{fig:nogo1}). Similarly, assume next that $m_{n_1}A \equiv p_{n_1} \pmod{2\pi}$ and $m_{n_1+n_2}A - m_{n_1}A \in (-d_2, d_2) \pmod{2\pi}$. Then $m_{n_2}A \in (-d_2, d_2) \pmod{2\pi}$ which again leads to a contradiction (see Figure \ref{fig:nogo2}). This means that for an arbitrary point $p_n$ on the unit circle such that $m_nA \equiv p_n \pmod{2\pi}$, we get an interval $(p_n-d_2,p_n+d_2)$ (understood as an angle on the unit circle) such that this interval cannot be visited later. As the whole unit circle is covered eventually, we obtain the expected contradiction.

\begin{figure}[ht!]
\centering
\includegraphics[width=0.9\textwidth]{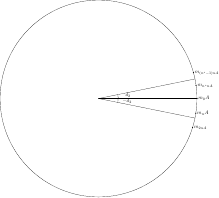}
\caption{Example of the excluded interval $(-d_2, d_2)$ around zero. Here $n^* = \lfloor\frac{2\pi}{|p_n|}\rfloor$, and we have visualized the points on the unit circle corresponding to the steps $0, n, 2n, (n^*-1)n$ and $n^*n$.}
\label{fig:nogo1}
\end{figure}

\begin{figure}[ht!]
\centering
\includegraphics[width=0.9\textwidth]{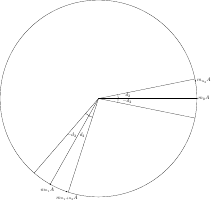}
\caption{Example of two excluded intervals and an angle $m_{n_1}A$.}
\label{fig:nogo2}
\end{figure}

\item
Consider first the case 
$
b  = 2\sin{\left(\frac{k}{l}\frac{\pi}{2}\right)}$,
where $\frac{k}{l}\in Q$. By Lemma \ref{lemma:eigenvalues}, the symmetric matrix $C$ defined by \eqref{eq:covariance-matrix} has non-negative eigenvalues, and thus $C$ is a covariance matrix of some random vector $(X_0,X_1,\ldots,X_{4l-1})$. Now it suffices to extend this vector to a process $X=(X_t)_{t\in\Z}$ by the relation $X_{4l+t} = X_t$. Indeed, it is straightforward to verify that $X$ has the covariance function $\gamma$. 

Assume next 
$
b  = 2\sin{\left(r\frac{\pi}{2}\right)}$,
where $r\in (0,1)\setminus Q$. We argue by contradiction and assume that there exists $k\in\mathbb{N}$, and vectors $t = (t_1, t_2,..., t_k)^T \in \mathbb{Z}^k$ and $a = (a_1, a_2,..., a_k)^T\in\mathbb{R}^k$ such that 

\begin{equation*}
\sum_{i, j = 1}^k a_i \gamma(t_i-t_j) a_j = -\epsilon\qquad\text{for some $\epsilon>0$,}
\end{equation*}
where $\gamma(\cdot)$ is the covariance function corresponding to the value $b$.
Since $Q$ is dense in $[0,1]$, it follows that there exists $(q_n)_{n\in\mathbb{N}} \subset Q$ such that $q_n\to r$. Denote the corresponding sequence of covariance functions with $\left(\gamma_n(\cdot)\right)_{n\in\mathbb{N}}$. By definition, 
\begin{equation*}
\sum_{i, j = 1}^k a_i \gamma_n(t_i-t_j) a_j \geq 0\qquad\text{for every $n$}.
\end{equation*}
On the other hand, continuity implies $\gamma_n(m) \to \gamma(m)$ for every $m$. This leads to  
\begin{equation*}
\lim_{n\to\infty} \sum_{i, j = 1}^k a_i \gamma_n(t_i-t_j) a_j = \sum_{i, j = 1}^k a_i \gamma(t_i-t_j) a_j =  -\epsilon
\end{equation*}
giving the expected contradiction.

\end{enumerate}

\end{proof}
\begin{rem}
Note that in the periodic case the covariance matrix $C$ defined by \eqref{eq:covariance-matrix} satisfies $rank(C)\leq 2$. Thus, in this case, the process $(X_t)_{t\in\Z}$ is driven linearly by only two random variables $Y_1$ and $Y_2$. In other words, we have 
$$
X_t = a_1(t)Y_1 + a_2(t)Y_2, \quad t\in\Z
$$
for some deterministic coefficients $a_1(t)$ and $a_2(t)$.
\end{rem}

\begin{proof}[Proof of Theorem \ref{thm:estimation}]
Suppose $\gamma$ satisfies \eqref{recursion}
and $r(m)\leq r(0)(1-\epsilon)$ for all $m\geq M$. By Theorem \ref{thm:main}, there exists $m^*\geq M$ such that 
$$
\gamma(m^*)\geq \gamma(0)\left(1-\frac{\epsilon}{2}\right).
$$ 
Furthermore, \eqref{recursion} implies \eqref{eq:r-gamma-ratio} for every $m$ such that $\gamma(m)\neq 0$. Now
$$
a_{m^*} = \frac{r(m^*)}{\gamma(m^*)} \leq \frac{r(0)(1-\epsilon)}{\gamma(0)\left(1-\frac{\epsilon}{2}\right)} < \frac{r(0)}{\gamma(0)} = a_0
$$
leading to a contradiction. Treating the case $r(m)\geq -r(0)(1-\epsilon)$ for all $m\geq M$ similarly concludes the proof.
\end{proof}

\bibliographystyle{plain}
\bibliography{bibli}

\begin{thebibliography}{1}

\bibitem{sole-et-al}
M.~Bahamonde, S.~Torres, and C.A. Tudor.
\newblock {ARCH} model with fractional {B}rownian motion.
\newblock {\em Statistics and Probability Letters}, 134:70--78, 2018.

\bibitem{brockwell}
P.J. Brockwell and R.A. Davis.
\newblock {\em Time Series: Theory and Methods}.
\newblock Springer Science \& Business Media, 2013.

\bibitem{Cheridito-Kawaguchi-Maejima-2003}
P.~Cheridito, H.~Kawaguchi, and M.~Maejima.
\newblock Fractional {O}rnstein-{U}hlenbeck processes.
\newblock {\em Electronic Journal of Probability}, 8:no. 3, 14 pp.
  (electronic), 2003.

\bibitem{hamilton1994time}
J.D. Hamilton.
\newblock {\em Time Series Analysis}, volume~2.
\newblock Princeton university press Princeton, 1994.

\bibitem{julia}
K.~Kubilius, Y.~Mishura, and K.~Ralchenko.
\newblock {\em Parameter Estimation in Fractional Diffusion Models}.
\newblock Springer, 2018.

\bibitem{viitasaari-ss}
L.~Viitasaari.
\newblock Representation of stationary and stationary increment processes via
  {L}angevin equation and self-similar processes.
\newblock {\em Statistics and Probability Letters}, 115:45--53, 2016.

\bibitem{vouti}
M.~Voutilainen, L.~Viitasaari, and P.~Ilmonen.
\newblock On model fitting and estimation of strictly stationary processes.
\newblock {\em Modern Stochastics: Theory and Applications}, 4(4):381--406,
  2017.

\end{thebibliography}

\end{document}